\documentclass[10pt]{amsart}
\usepackage{amssymb}
\usepackage{amsmath,amssymb,amsfonts,amsthm,graphics,
latexsym, amscd, amsfonts, epsfig}
\usepackage{mathrsfs}
\usepackage{color}
\usepackage{eucal}
\usepackage{eufrak}
\usepackage[all]{xypic}
\usepackage{xspace}

\theoremstyle{plain}
\newtheorem{theorem}{Theorem}
\newtheorem{corollary}{Corollary}

\newtheorem{lemma}{Lemma}

\theoremstyle{definition}
\newtheorem{definition}{Definition}

\theoremstyle{example}

\theoremstyle{remark}

\numberwithin{equation}{section}

\begin{document}

\title[Combinatorics of RNA Structures with Pseudoknots]
      {Combinatorics of RNA Structures with Pseudoknots}
\author{Emma Y. Jin, Jing Qin and Christian M. Reidys$^{\,\star}$}
\address{Center for Combinatorics, LPMC \\
         Nankai University  \\
         Tianjin 300071\\
         P.R.~China\\
         Phone: *86-22-2350-6800\\
         Fax:   *86-22-2350-9272}
\email{duck@santafe.edu}
\thanks{}
\date{April 2007}
\begin{abstract}
In this paper we derive the generating function of RNA structures with
pseudoknots. We enumerate all $k$-noncrossing RNA pseudoknot structures
categorized by their maximal sets of mutually intersecting arcs.
In addition we enumerate pseudoknot structures over circular RNA.
For $3$-noncrossing RNA structures and RNA secondary structures
we present a novel $4$-term recursion formula and a $2$-term recursion,
respectively.
Furthermore we enumerate for arbitrary $k$ all $k$-noncrossing, restricted
RNA structures i.e.~$k$-noncrossing RNA structures without $2$-arcs
i.e.~arcs of the form $(i,i+2)$, for $1\le i\le n-2$.
\end{abstract}
\maketitle
{{\small
}}


\section{Introduction}


In this paper we study the combinatorics of helical structures of
RNA sequences. RNA is described by its primary sequence of nucleotides
{\bf A}, {\bf G}, {\bf U} and {\bf C} together with the Watson-Crick
({\bf A-U}, {\bf G-C}) and ({\bf U-G}) base pairing rules specifying
which pairs of nucleotides can potentially form bonds. Subject to these
single stranded RNA form helical structures.
The function of many RNA sequences depends on their structures. Therefore
it is important to understand RNA structure in the context of studying the
function of biological RNA as well as in the design process of artificial
RNA structures. Since RNA is capable of catalytic activity, for instance
RNA ribozymes can cleave other RNA molecules, it is believed that RNA may
have been instrumental for early evolution, before Proteins emerged.
A particularly well-studied sub-class of RNA structures, consisting of
planar graphs are the RNA secondary structures.
Their combinatorics was pioneered by Waterman {\it et.al.} in a series
of seminal papers
\cite{Penner:93c,Waterman:79a,Waterman:78a,Waterman:94a,Waterman:80}.
RNA secondary structures are coarse grained structures and systematic
prediction of the full three dimensional structures, the tertiary
structures seems at present time to be out of reach.
It was shown in \cite{Waterman:86} that the prediction of secondary
structures can be obtained in polynomial time and their combinatorics,
specifically the existence of recursion relations is the key for all
folding algorithms \cite{Zuker:79b,Schuster:98}. Over the last two
decades a variety of prediction algorithms, based on minimum free
energy \cite{Zuker:79b,Waterman:86,Bauer:96}, kinetic folding
\cite{Tacker:94a} or the partition function \cite{McCaskill:90a}
for RNA secondary structures has been derived.

An increasing number of experimental findings, as well as results from
comparative sequence analysis imply that there exist additional types
of interactions between RNA nucleotides \cite{Westhof:92a}. These bonds
are called pseudoknots and occur in functional RNA like for instance
RNAseP \cite{Loria:96a} as well as ribosomal RNA \cite{Konings:95a}.
RNA pseudoknots are conserved also in the catalytic core of group I
introns. In plant viral RNAs pseudoknots mimic tRNA structure and in
{\it in vitro} RNA evolution \cite{Tuerk:92} experiments have produced
families of RNA structures with pseudoknot motifs, when binding HIV-1
reverse transcriptase. In addition important mechanisms like ribosomal
frame shifting \cite{Chamorro:91a} also involve pseudoknot interactions.
As a result RNA pseudoknot structures have drawn over the last years a
lot of attention \cite{Science:05a}. Several folding algorithms
\cite{Rivas:99a,Uemura:99a,Akutsu:00a,Lyngso:00a}
have been developed which include certain families of pseudoknots.
The prediction problem in general is (although we have not seen formal
proof) believed to be NP-hard. In difference to RNA secondary structures
a recursive enumeration for pseudoknot RNA is believed to be
non-trivial but nevertheless of vital importance for prediction algorithms.
Intuitively if bonds can cross it is much harder to enumerate since
structural elements can now interact and as a result a structure cannot
be straighforwardly decomposed into independent sub-structures.
Little is known with respect to the combinatorics of pseudoknot
RNA structures.
Stadler {\it et al.} \cite{Stadler:99a} suggested a
classification of their knot-types based on a notion of inconsistency
graphs and provided an upper bound for a certain class of
pseudoknots (our 3-noncrossing RNA structures).

In this paper we introduce a novel approach for the enumeration of
RNA structures. Based on new concepts in enumerative combinatorics
\cite{Chen:07a,Gessel:92a} we use a method which has the
potential to
offer insight also into other lattice structure concepts. To be precise
Chen~{\it et.al.} have shown in \cite{Chen:07a} that there is a bijection
between certain types of matchings and walks inside Weyl-chambers.
This bijection is obtained via his construction of oscillating
tableaux i.e.~families of Young diagrams in which any two consecutive
shapes differ by exactly one square. The corresponding
walks can then be enumerated via determinant formulas derived from a
reflection principle due to Gessel and Zeilberger
\cite{Gessel:92a} and
Lindstr\"om \cite{Lindstroem:73a}. The key idea behind the reflection
principle is that walks which hit the wall of a Weyl-chamber can be
reflected. The original (unreflected) and the reflected walk
cancel themselves leaving just the walks that never hit a wall.
Crucial for its applicability are restrictive symmetry assumptions
since the reflected walk has to be of the same type and, more
importantly, the reflection itself can occur at any step.
These symmetries are non-existent in walks corresponding to RNA
structures. However, our derivation of the generating function of RNA
structures is based on these symmetric walks. The key idea is to
introduce the asymmetries of RNA structures into the symmetric walks
using a certain involution idea. We believe that our particular strategy
can be applied for the enumeration of further structure classes.
As a result we have tried to keep this paper self contained.

Our main result is the enumeration of all RNA structures. We classify RNA
structures by their specific crossing types under the assumption that all
base pairs can occur. For arbitrary but fixed $k$ we enumerate all RNA
structures with no $k$-set of mutually intersecting bonds.
In case of $k=2$, our results
reduce to noncrossing structures, i.e.~RNA secondary structures
\cite{Waterman:94a} and Waterman's formula for the number of
RNA secondary structures with exactly $k$ bonds is derived directly.
The case $k=3$ coincides with Stadler's bi-secondary
structure \cite{Stadler:99a}. We obtain from the generating
function a novel $4$-term recursion formula for RNA structures of length
$n$ with no $3$-set of mutually intersecting bonds  and
having $\ell$ isolated vertices. We believe that this recursion is
the key for developing new prediction algorithms for RNA structures.
Also we derive the generating function for circular RNA structures
i.e.~for sequences where the bond between $1$ and $n$ is considered part
of the primary sequence. Finally we enumerate restricted RNA structures,
i.e.~structures in which two interacting nucleotides have at least
distance $3$.

The paper is structured as follows. We will begin by introducing several
important combinatorial concepts needed for our derivations. Young
tableaux, oscillating Young diagrams, RSK algorithm, Weyl-chambers and
the reflection principle. We discuss these concepts, illustrate all key
ideas and give the corresponding proofs in the Appendix. Our derivation is
obtained in three steps. First (Theorem~\ref{T:tableaux}) we show that
each structure (represented as a $k$-noncrossing digraph)
corresponds uniquely to a walk starting and ending at $(k-1,k-2,
\dots,1)$ in $\mathbb{Z}^{k-1}$ and which never touches a wall of the
Weyl-chamber $C_0$.
Secondly we apply the reflection principle (Theorem~\ref{T:reflect}) in
order to count the symmetric walks that remain in the interior of $C_0$.
Thirdly (Theorem~\ref{T:cool1}) we incorporate the
specific properties of RNA into these symmetric walks and compute the
generating function of $k$-noncrossing RNA structures. We show how our
results relate to known formulas of RNA secondary structures for which
we present a two new term recursion formula. For $3$-noncrossing RNA
structures we give a novel $4$-term recursion formula. We finally
generalize our strategy (Theorem~\ref{T:cool2}) and enumerate restricted
RNA structures.

\section{From structures to walks and back}


Let us begin by illustrating the concept of RNA structures. Suppose we are
given the primary sequence
$$
{\bf A}{\bf A}{\bf C}{\bf C}{\bf A}{\bf U}{\bf G}{\bf U}{\bf G}{\bf G}
{\bf U}{\bf A}{\bf C}{\bf U}{\bf U}{\bf G}{\bf A}{\bf U}{\bf G}{\bf G}
{\bf C}{\bf G}{\bf A}{\bf C}  \ .
$$
Structures are combinatorial graphs over the labels of the nucleotides
of the primary sequence. These graphs can be represented in several ways.
In Figure $1$ we represent a particular structure with loop-loop
interactions in two ways:
first we display the structure as a planar graph
and secondly as a diagram, where the bonds are drawn as arcs in the
positive half-plane.
\begin{figure}[ht]
\centerline{%
\epsfig{file=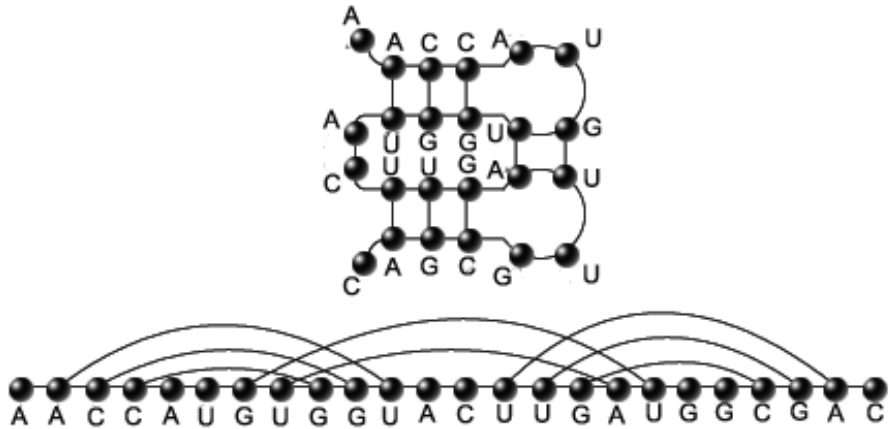,width=0.9\textwidth}\hskip15pt
 }
\caption{\small Two representations of RNA structures, planar graphs (top)
and diagrams (bottom)}
\label{fig:1}
\end{figure}

In the following we will consider structures as diagram representations of
digraphs.
A digraph $D_n$ is a pair of sets $V_{D_n},E_{D_n}$, where $V_{D_n}=
\{1,\dots,n\}$ and $E_{D_n}\subset \{(i,j)\mid 1\le i< j\le n\}$.
$V_{D_n}$ and $E_{D_n}$ are called vertex and arc set, respectively.
A $k$-noncrossing digraph, $G_{k,n}$, is a digraph in which all
vertices have degree $\le 1$ and which does not contain a $k$-set of
arcs that are mutually intersecting, i.e.
\begin{eqnarray}
\not\exists\,
(i_{r_1},j_{r_1}),(i_{r_2},j_{r_2}),\dots,(i_{r_k},j_{r_k});\quad & &
i_{r_1}<i_{r_2}<\dots<i_{r_k}<j_{r_1}<j_{r_2}<\dots<j_{r_k} \ .
\end{eqnarray}
The set of all $k$-noncrossing digraphs $G_{k,n}$ is denoted by
$\mathcal{G}_{n,k}$ and we set ${\sf G}_{n,k}=\vert\mathcal{G}_{n,k}
\vert$.
The (formal) direction of the edges will have procedural convenience when
we map a $k$-noncrossing digraph into an oscillating tableaux
(Theorem~\ref{T:tableaux}).
We will represent digraphs as a diagrams (Figure~\ref{fig:1}) by
representing the vertices as
integers on a line and connecting any two adjacent vertices by an arc in
the upper-half plane. The direction of the arcs is implicit in the
linear ordering of the vertices and accordingly omitted.
\begin{definition}\label{D:rna}
An RNA structure (of pseudoknot type $k-2$), ${S}_{k,n}$, is a digraph
in which all vertices have degree $\le 1$, that does not contain a
$k$-set of mutually intersecting arcs and $1$-arcs, i.e.~arcs of the
form $(i,i+1)$, respectively.
We denote the number of RNA structures by ${\sf S}_k(n)$ and the number
of RNA structures with exactly $\ell$ isolated vertices by ${\sf S}_k(n,
\ell)$, respectively.
We call an RNA structure restricted iff it does not contain any $2$-arcs,
i.e.~an arc of the form $(i,i+2)$.
\end{definition}
\begin{figure}[ht]
\centerline{%
\epsfig{file=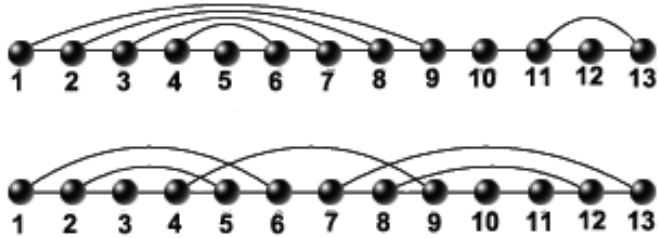,width=0.7\textwidth}\hskip15pt
 }
\caption{\small RNA structures represented as
diagrams, i.e.~arcs over $\{1,\dots,n\}$ in the upper half-plane.
$k$-noncrossing digraphs are precisely those which have no $k$-set
of mutually intersecting arcs. $2$-noncrossing diagrams without $1$-arcs
(top) correspond to secondary structures. $3$-noncrossing diagrams without
$1$-arcs (bottom) represent RNA structures with pseudoknots. }
\label{fig:2}
\end{figure}
We derive the enumeration of RNA structures in three steps.
First we establish a bijection from $k$-noncrossing digraphs into a
certain class of walks.
Secondly we will use the reflection principle in order to count these
walks. Thirdly we enumerate all walks subject to specific conditions
recruiting a certan involution idea.
Let us first discuss two basic concepts needed for our arguments.

{\bf Young tableaux and the RSK algorithm.}
A Young diagram (shape) is a collection of squares arranged in
left-justified
rows with weakly decreasing number of boxes in each row. A Young tableau
is a filling of the squares by numbers which is weakly decreasing
in each row and strictly decreasing in each column. A tableau is called
standard if each entry occurs exactly once. An oscillating tableau is a
sequence $\varnothing=\mu^{0}, \mu^{1},\ldots,\mu^{n}=
\varnothing$ of standard Young diagrams, such that for
$1\le i \le n$, $\mu^{i}$
is obtained from $\mu^{i-1}$ by either adding one square or removing
one square. For instance the sequence
\begin{figure}[ht]
\centerline{%
\epsfig{file=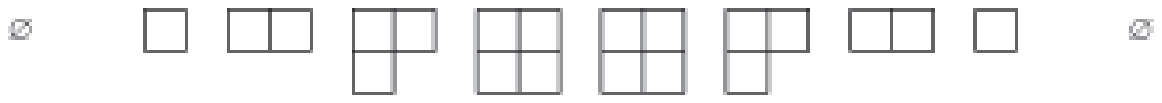,width=0.75\textwidth}\hskip15pt
 }
\caption{1}
\end{figure}

is an oscillating tableaux.

The RSK algorithm is a process of row-inserting elements into a tableau.
Suppose we want to insert $k$ into a standard Young tableau $\lambda$.
Let $\lambda_{i,j}$ denote the element in the $i$th row and $j$th
column of the Young tableau. Let $i$ be the largest integer such that
$\lambda_{1,i-1}\le k$. (If $\lambda_{1,1}>k$, then $i=1$.) If
$\lambda_{1,i}$ does not exist, then simply add $k$ at the end of the
first row. Otherwise, if $\lambda_{1,i}$ exists, then replace
$\lambda_{1,i}$ by $k$. Next insert $\lambda_{1,i}$ into the second
row following the above procedure and continue until an element is
inserted at the end of a row. As a result we obtain a new standard
Young tableau with $k$ included. For instance inserting the number
sequence $5,2,4,1,6,3$ starting with an empty shape yields the following
sequence of standard Young tableaux:
\par\medskip

\par\medskip

\par\medskip

\par\medskip
\begin{figure}[ht]
\centerline{%
\epsfig{file=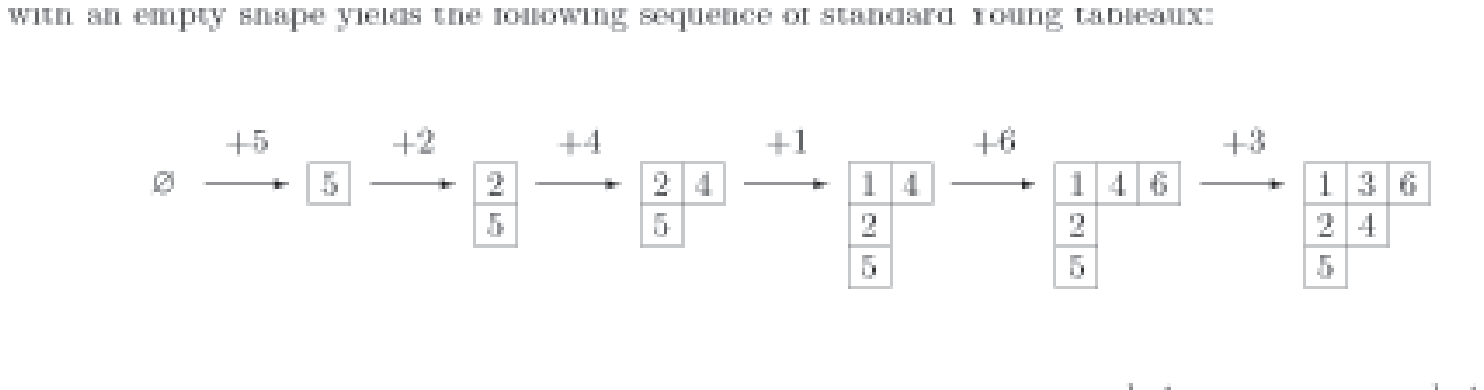,width=0.75\textwidth}\hskip15pt
 }
\caption{} \label{fig:4}
\end{figure}

{\bf Symmetry groups and Weyl-chambers.} We consider the lattice
$\mathbb{Z}^{k-1}$ and walks in $\mathbb{Z}^{k-1}$ having the steps
$s$ contained in $\{\pm e_i,0\mid 1\le i\le k-1\}$, where $e_i$ denotes
the $i$th unit vector. That is for
$a,b\in \mathbb{Z}^{k-1}$ a walk from $a$ to $b$, $\gamma_{a,b}$, of
length $n$ is an $n$ tuple $(s_1,\dots,s_n)$ where $s_i\in \{\pm
e_i,0 \mid 1\le i\le k-1\}$ such that $b=a+\sum_{h=1}^ns_h$. We set
$\gamma_{a,b}(s_r)=a+\sum_{h=1}^rs_h\in \mathbb{Z}^{k-1}$ i.e.~the
element at which the walk $(s_1,\dots,s_{r})$ resides at step $r$.
The symmetric group $S_{k-1}$ acts on $\mathbb{Z}^{k-1}$ via $\sigma
(x_i)_{1\le i\le k-1}= (x_{\sigma^{-1}(i)})_{1\le i\le k-1}$. We set
$E_{k-1}=\langle \epsilon_i\mid 1\le i\le k-1\rangle$, where
$\epsilon_i(x_1,\dots,
x_i,\dots,x_{k-1})=(x_1,\dots,-x_i,\dots,x_{k-1})$. As shown in the
Appendix $\{\epsilon_i\sigma\mid \sigma \in S_{k-1},\, \epsilon_i
\in E_{k-1}\}$ carries a natural group structure via $(\epsilon_i
\sigma) \cdot (\epsilon_j\sigma')=\epsilon_i \sigma \epsilon_j
\sigma^{-1}\sigma\sigma'=\epsilon_i\epsilon_{
\sigma^{-1}(j)}\sigma\sigma'$. This group, denoted by ${\sf
B_{k-1}}$, is generated by $M_{k-1}=\{\epsilon_1,\rho_j\mid 2\le
j\le k-1\}$, where $\rho_j=(j,j-1)$, i.e.~$\rho$ transposes the
coordinates $x_{j-1}$ and $x_{j}$. By definition ${\sf B}_{k-1}$
acts on the set
\begin{equation}
\Delta_{k-1}=\{\pm e_i \mid 1\le i\le k-1\}\, \cup \,
\{e_i \pm e_j \mid 1\le i,j\le k-1\}
\end{equation}
and we call $\Delta_{k-1}$ the set of roots. We observe that there
exists a
bijection between $\Delta_{k-1}'=\{e_1, e_j-e_{j-1} \mid 2\le j\le k-1\}$
and the set of generators $M_{k-1}$ which maps each root
$\alpha\in\Delta_{k-1}'$ into
a corresponding reflection (in particular: ${\sf B}_{k-1}$ is generated
by reflections)
\begin{equation}\label{E:Co}
\{ e_1, e_j-e_{j-1}\mid 2\le j\le k-1 \} \longrightarrow
\{ \epsilon_1,\rho_j\mid 2\le j\le k-1\} , \quad \alpha \mapsto
\left(\beta_\alpha= x \mapsto x-2 \frac{\langle \alpha,x\rangle }
          {\langle\alpha,\alpha\rangle}\right)
\end{equation}
where $\langle x,x'\rangle$ denotes the standard scalar product in
$\mathbb{Z}^{k-1}$. It is clear that $\Delta_{k-1}'$ is a basis of
$\mathbb{Z}^{k-1}$. We refer to the sub spaces $\langle e_i\rangle$
for $1\le i\le k-1$ and $\langle e_j-e_{j-1}\rangle$ for $2\le j\le
k-1$ as walls. A ${\sf B}_{k-1}$-chamber is defined as the set of
$x\in \mathbb{Z}^{k-1}$ with the property that $\langle
\alpha,x\rangle \ge 0$ for all $\alpha \in \Delta_{k-1}$.
We denote the Weyl chamber
\begin{equation}\label{E:C_0}
C_0=\{x\in\mathbb{Z}^{k-1}\mid 0<x_{k-1}<x_{k-2}<\dots<x_1\} \ .
\end{equation}
For RNA secondary structures we have $k-1=1$, and
${\sf B}_1=E_1=\{\epsilon_1, 1\}$ and $\Delta'=\{e_1\}$.
For $3$-noncrossing RNA we have $k-1=2$ and
${\sf B}_2=E_2 \rtimes_\varphi S_2\cong D_4$ (where $\varphi:S_2\to
{\sf Aut}(E_2)$) is the dihedral group of order $8$.

The following theorem is the first step for the enumeration of RNA
structures. It will allow to interpret a certain class of digraphs
as walks in $\mathbb{Z}^{k-1}$ which remain in the interior of the Weyl
chamber $C_0$. The result is due to Chen~{\it et al} \cite{Chen:07a},
where it is formulated for matchings. The original bijection between
oscillating tableaux and matchings is due to Stanley and was
generalized Sundaram \cite{Sundaram:90a}.
We give a proof of Theorem~\ref{T:tableaux} in the Appendix.
\begin{theorem}\cite{Chen:07a}\label{T:tableaux}
There exists a bijection between $k$-noncrossing digraphs and walks
of length $n$ in $\mathbb{Z}^{k-1}$ which start and end at $a=(k-1,k-2,
\dots,1)$ having steps $0,\pm e_i$, $1\le i\le k-1$ such that $0<
x_{k-1}<\dots < x_1$ at any step.
I.e.~we have a bijection
\begin{equation}\label{E:bij3}
\mathcal{G}_{n,k} \longrightarrow
\{\gamma_{a,a}\mid \gamma_{a,a} \,
                    \text{\it remains inside the Weyl-chamber $C_0$}\} \ ,
\end{equation}
where $\mathcal{G}_{n,k}$ denotes the set of $k$-noncrossing digraphs of
length $n$.
\end{theorem}
\begin{figure}[ht]
\centerline{%
\epsfig{file=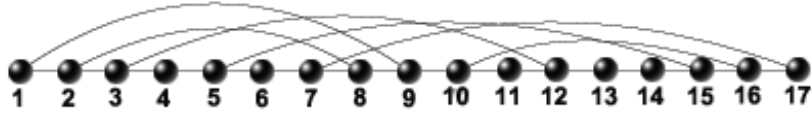,width=0.75\textwidth}\hskip15pt
 }
\caption{\small A $5$-noncrossing digraph. In the text we show how to
derive from this digraph an oscillating tableau and subsequently the
corresponding walk $\gamma_{a,a}$ in $\mathbb{Z}^4$.}
\label{fig:3}
\end{figure}
\begin{figure}[ht]
\centerline{%
\epsfig{file=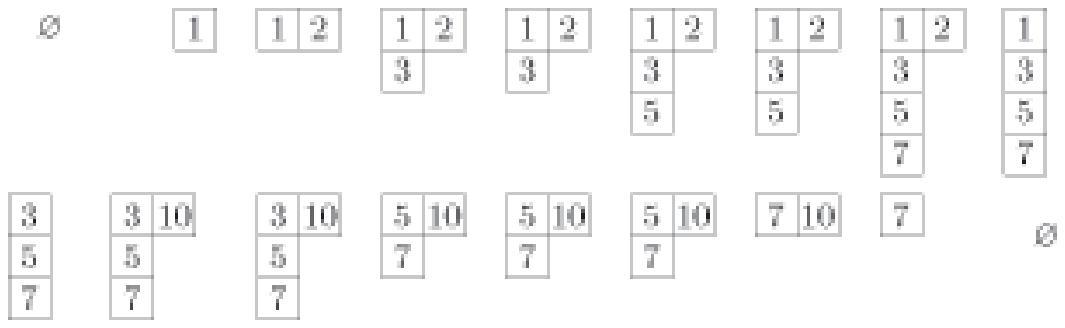,width=0.75\textwidth}\hskip15pt
 }
\caption{1} \label{fig:5}
\end{figure}

The $5$-noncrossing digraph in Figure~\ref{fig:3} corresponds to a
oscillating tableau as follows:
from right to left start at vertex $17$, which is a terminus. The
corresonding origin is $7$ which is inserted via the RSK algorithm
into the empty shape. Next insert the origin corresponding to $16$ and
$15$, respectively. At $14$ and $13$ nothing happens since they are
isolated vertices. At $12$ origin $3$ is inserted and $11$ is isolated.
$10$ is an origin of an arc and accordingly removed from the tableaux.
It is now clear how to proceed. The walk $\gamma_{a,a}$ is obtained
from the tableau as follows: its $x_i$-coordinate is the number of
squares in the $i$-th column, i.e.~$\gamma_{a,a}$ is given by
{\small
\begin{align*}
&(4,3,2,1),(5,3,2,1),(6,3,2,1),(6,4,2,1),(6,4,2,1),(6,4,2,1),(6,4,3,1),
(6,4,3,1),(5,4,3,1),\\
&(5,4,3,2),(6,4,3,2),(6,4,3,1),(6,4,3,1),(6,4,2,1),(6,4,2,1),(6,3,2,1),
(5,3,2,1),(4,3,2,1) \ .
\end{align*}}
We show in the appendix in detail why this is a bijection.

We next discuss the reflection principle. The key idea is to count
walks that remain in the interior of a Weyl chamber by counting all
walks. Then one utilizes the fact that all walks that touch a wall
at some step can be paired and eventually cancel themselves in the
enumeration. The particular way to obtain this pairing is by reflecting
the walk at the corresponding wall. The following observation is essential
for the reflection principle, formulated in Theorem~\ref{T:reflect} below.
\begin{lemma}\label{L:touch}
Let  $\Delta_{k-1}'=\{ e_1, e_j-e_{j-1}\mid 2\le j\le k-1 \}$. Then
every walk starting at some lattice point in the interior of $C$
having steps $\pm e_i,0$ that crosses from inside ${C}$ into outside
${C}$ touches a subspace $\langle e_i- e_{i-1}\mid 2\le i\le
k-1\rangle$ or $\langle e_i\mid 1\le i\le k-1\rangle$.
\end{lemma}
Let $\Gamma_n(a,b)$ be the number of walks $\gamma_{a,b}$. For
$a,b\in C_0$ (eq.~(\ref{E:C_0}))
let $\Gamma_n^+(a,b)$ denote the number of walks $\gamma_{a,b}$ that never
touch a wall, i.e.~remain in the interior of $C_0$. Finally for
$a,b\in \mathbb{Z}^{k-1}$, let $\Gamma_n^-(a,b)$ denote the number of walks
$\gamma_{a,b}=(s_1,\dots,s_n)$ that hit a wall at some step $s_r$.
$\ell(\beta)$ denotes the length of $\beta\in{\sf B}_{k-1}$.
For $a=b=(k-1,\dots,1)$ we have according to Theorem~\ref{T:tableaux}
\begin{equation}
\Gamma_n^+(a,a) = {\sf G}_{n,k}\ ,
\end{equation}
where ${\sf G}_{n,k}=\vert \mathcal{G}_{n,k}\vert$.
\begin{theorem}{\bf (Reflection-Principle)}\cite{Gessel:92a}
\label{T:reflect}
Suppose $a,b\in C_0$, then we have
\begin{equation}
\Gamma_n^+(a,b)=
           \sum_{\beta \in {\sf B}_{k-1}}(-1)^{\ell(\beta)} \,
                  \Gamma_n(\beta(a),b) \ .
\end{equation}
\end{theorem}
Theorem \ref{T:reflect} allows us to compute the exponential
generating function for $\Gamma_{n}^{+}(a,b)$, which is the number of
walks from $a$ to $b$, that remain in the interior of $C_{0}$
\cite{Grabiner:93a}.
\begin{lemma}\label{L:walks}\cite{Grabiner:93a}
Let $I_{r}(2x)=\sum_{j \ge 0}x^{2r+j}/{j!(r+j)!}$ be the
hyperbolic Bessel function of the first kind of order $r$. Then the
generating functions for the numbers of $k$-noncrossing digraphs of
length $n$ and for $k$-noncrossing digraphs of length $n$ without
isolated points, $\Gamma_n^+(a,b)$ and ${\Gamma'}_n^+(a,b)$ are given by
\begin{eqnarray}
\sum_{n \ge 0}\Gamma_{n}^{+}(a,b)\frac{x^{n}}{n!}& =& e^{x}\det
[I_{a_{i}-b_{j}}(2x)-I_{a_{i}+b_{j}}(2x)]|_{i,j=1}^{k-1} \\
\sum_{n \ge 0}{\Gamma'}_{n}^{+}(a,b)\frac{x^{n}}{n!}& =& \det
[I_{a_{i}-b_{j}}(2x)-I_{a_{i}+b_{j}}(2x)]|_{i,j=1}^{k-1} \ .
\end{eqnarray}
\end{lemma}
Now we can combine our results in order to enumerate $k$-noncrossing
digraphs using the bijection between digraphs and walks.
Theorem~\ref{T:tableaux} implies that the number of $k$-noncrossing
digraphs is equal to $\Gamma^+_{k}(a,a)$, the number of lattice walks
in $\mathbb{Z}^{k-1}$ of length $n$ that remain in the interior of
$C_{0}$ from $(k-1,\ldots 1)$ to itself with steps $0,\pm e_{i}$, $1
\le i \le k-1$. According to Lemma~\ref{L:walks} the generating
functions for walks with steps $e_i=\pm 1,0$ and $e_i=\pm 1$,
where $1\le i\le k-1$ are given by
\begin{equation}\label{E:gen}
e^{x}\det_{}[I_{i-j}(2x)-I_{i+j}(2x)]|_{i,j=1}^{k-1} \quad
\text{\rm and}\quad \det[I_{i-j}(2x)-I_{i+j}(2x)]|_{i,j=1}^{k-1} \ ,
\end{equation}
respectively.
Let $f_{k}(n,\ell)$ denote the number of $k$-noncrossing digraphs with
$\ell$ isolated points. Then
\begin{align}\label{E:ww0}
f_{k}(n,\ell)& ={n \choose \ell} f_{k}(n-\ell,0) \\
\label{E:ww1}
\det[I_{i-j}(2x)-I_{i+j}(2x)]|_{i,j=1}^{k-1} &=
\sum_{n\ge 1} f_{k}(n,0)\cdot\frac{x^{n}}{n!} \\
\label{E:ww2}
e^{x}\det[I_{i-j}(2x)-I_{i+j}(2x)]|_{i,j=1}^{k-1}
&=(\sum_{\ell \ge 0}\frac{x^{\ell}}{\ell!})(\sum_{n \ge
1}f_{k}(n,0)\frac{x^{n}}{n!})=\sum_{n\ge 1}
\left\{\sum_{\ell=0}^nf_{k}(n,\ell)\right\}\cdot\frac{x^{n}}{n!} \ .
\end{align}
In particular we obtain for
$k=2$ and $k=3$
\begin{equation}\label{E:2-3}
f_2(n,\ell)  =  \binom{n}{\ell}\,C_{(n-\ell)/2}\quad
\text{\rm and}\quad  f_{3}(n,\ell)=
{n \choose \ell}\left[C_{\frac{n-\ell}{2}+2}C_{\frac{n-\ell}{2}}-
      C_{\frac{n-\ell}{2}+1}^{2}\right] \ ,
\end{equation}
where $C_m$ denotes the $m$th Catalan number.
\section{RNA structures}\label{S:RNA}

In this section we derive the generating function for RNA structures.
The successful strategy consists in counting the ``wrong'' object
``multiple'' times.
To be precise we will enumerate all $k$-noncrossing digraphs with $j$
$1$-arcs by relating them to a {\it family} whose cardinality
we can easily compute.
We denote the number of RNA structures with exactly $\ell$ isolated
vertices by ${\sf S}_k(n,\ell)$. Suppose $k\ge 2$
and let $\mathcal{G}^{}_{n,k}(\ell,j)$ be the set of all $k$-noncrossing
digraphs having exactly $\ell$ isolated points and exactly $j$ $1$-arcs,
where a $1$-arc is an arc of the form $(i,i+1)$, $1\le i\le n-1$. Setting
$
{\sf G}_{k}(n,\ell,j)=\vert \mathcal{G}_{n,k}(\ell,j)\vert
$,
we have in particular ${\sf S}_k(n,\ell)={\sf G}_{k}(n,\ell,0)$.
\begin{theorem}\label{T:cool1}
Let $k\in\mathbb{N}$, $k\ge 2$, let $C_m$ denote the $m$-th Catalan number
and $f_k(n,\ell)$ be the number of $k$-noncrossing digraphs over $n$
vertices with exactly $\ell$ isolated vertices. Then the
number of RNA structures with $\ell$ isolated vertices,
${\sf S}_k(n,\ell)$, is given by
\begin{equation}\label{E:da}
{\sf S}_k(n,\ell) = \sum_{b=0}^{(n-\ell)/2}
                   (-1)^b\binom{n-b}{b}f_k(n-2b,\ell)  \  ,
\end{equation}
where $f_k(n-2b,\ell)$ is given by the generating function in
eq.~{\rm (\ref{E:ww1})}.
Furthermore the number of $k$-noncrossing RNA structures, ${\sf S}_k(n)$
is
\begin{equation}\label{E:sum}
{\sf S}_k(n)
=\sum_{b=0}^{\lfloor n/2\rfloor}(-1)^{b}{n-b \choose b}
\left\{\sum_{\ell=0}^{n-2b}f_{k}(n-2b,\ell)\right\}
\end{equation}
where $\{\sum_{\ell=0}^{n-2b}f_{k}(n-2b,\ell)\}$ is given by the
generating function in eq.~{\rm (\ref{E:ww2})}.
\end{theorem}
\begin{proof}
We first prove
\begin{equation}\label{E:form1}
\sum_{j\ge b}\binom{j}{b}\, {\sf G}_k(n,\ell,j)=
\binom{n-b}{b}
\, f_k(n-2b,\ell) \ .
\end{equation}
For this purpose we construct a family $\mathcal{F}$ of
$\mathcal{G}_{n,k}$-digraphs, having exactly $\ell$ isolated
points and having at least $b$ $1$-arcs as follows:
select {\sf (a)} $b$ $1$-arcs, and {\sf (b)} an arbitrary
$k$-noncrossing digraph with exactly $\ell$ isolated points
over the remaining $n-2b$ vertices. Let $\mathcal{F}$ be the
resulting family of digraphs. \\
{\it Claim $1$.} Each element $\theta\in \mathcal{F}$ is contained in
$\mathcal{G}_{n,k}(\ell,j)$ for some $j\ge b$.\\
To prove this we observe that a $1$-arc cannot cross any other arc,
i.e.~cannot be contained in a set of mutually crossing arcs. As a
result for $k\ge 2$ our construction generates digraphs that are
$k$-noncrossing. Clearly $\theta$ has exactly $\ell$ isolated
vertices and in step {\sf (b)} we potentially derive additional
$1$-arcs, whence $j\ge b$.\\
{\it Claim $2$.}
\begin{equation}\label{E:fam}
\vert \mathcal{F}\vert =\binom{n-b}{b}f_k(n-2b,\ell)\ .
\end{equation}
Let $\lambda(n,b)$ denote the number of ways to select $b$ $1$-arcs
over $\{1,\dots,n\}$. We observe that $\lambda(n,b)=\binom{n-b}{b}$.
Identifying the two incident vertices of an $1$-arc we conclude that we
can choose the $b$ $1$-arcs in $\binom{n-b}{b}$ ways.
Obviously, $\ell$ isolated vertices can be obtained in $\binom{n-2b}
{\ell}$ different ways and it remains to select an arbitrary
$k$-noncrossing digraph with exactly $\ell$ isolated points over $n-2b$
vertices. The number of those is given by $f(n-2b,\ell)$ which we can
compute via Lemma~\ref{L:walks}, whence eq.~(\ref{E:fam}) and Claim $2$ is
proved.\\
In view of the fact that any of the $k$-noncrossing digraphs can introduce
additional $1$-arcs we set
$$
\mathcal{F}(j)=\{\theta\in\mathcal{F} \mid \theta\ \text{\rm has
                     exactly $j$ $1$-arcs} \} \ .
$$
Obviously, $\mathcal{F}=\dot\bigcup_{j \ge b}\mathcal{F}(j)$.
Suppose $\theta \in \mathcal{F}(j)$. According to Claim $1$, $\theta\in
\mathcal{G}_{n,k}(\ell,j)$ and furthermore $\theta$ occurs with
multiplicity $\binom{j}{b}$ in $\mathcal{F}$ since by construction any
$b$-element subset of the $j$ $1$-arcs is counted respectively in
$\mathcal{F}$. Therefore we have
\begin{equation} \label{E:well}
\vert \mathcal{F}(j)\vert = \binom{j}{b}{\sf G}_k(n,\ell,j)
\end{equation}
and
\begin{eqnarray*}
\sum_{j \ge b}\binom{j}{b}{\sf G}_k(n,\ell,j)  =
\sum_{j \ge b} \vert \mathcal{F}(j)\vert =
\binom{n-b}{b} f_k(n-2b,\ell) \ ,
\end{eqnarray*}
whence eq.~(\ref{E:form1}). We next set
$F_k(x)=\sum_{j\ge 0}{\sf G}_k(n,\ell,j) \, x^j$.
Taking the $b$-th derivative and let $x=1$ we obtain
\begin{eqnarray}\label{E:u}
\frac{1}{b!} F_k^{(b)}(1) =
\sum_{j\ge b}\binom{j}{b}{\sf G}_k(n,\ell,j)1^{j-b}
 \ .
\end{eqnarray}
Claim~$2$ provides an interpretation of
the r.h.s.~of eq.~(\ref{E:u})
\begin{equation}
\sum_{j\ge b}\binom{j}{b}{\sf G}_k(n,\ell,j)\, 1^{j-b}=
\binom{n-b}{b}f_k(n-2b,\ell) \ .
\end{equation}
In order to connect $F_k(x)$ and $\frac{1}{b!} F^{(b)}(1)$ we consider the
Taylor expansion of $F_k(x)$ at $x=1$ and compute
\begin{eqnarray*}
F_k(x)  =  \sum_{b\ge 0} \frac{1}{b!} \, F^{(b)}(1) (x-1)^b
        =  \sum_{b=0}^{(n-\ell)/2}\binom{n-b}{b} f_k(n-2b,\ell)
              (x-1)^b \ .\\
\end{eqnarray*}
In view of ${\sf S}_k(n,\ell)={\sf G}_k(n,\ell,0)$ is
the constant term of $F_k(x)$, i.e.~$F_k(0)$, whence
\begin{equation}\label{E:y}
{\sf S}_k(n,\ell)=\sum_{b=0}^{(n-\ell)/2}\, (-1)^b\,
                  \binom{n-b}{b} f_k(n-2b,\ell) \ .
\end{equation}
It remains to prove eq~(\ref{E:sum}). Summing over all possible values
of isolated vertices, we get
\begin{align*}
S_{k}(n)&=\sum_{\ell=0}^{n}\sum_{b=0}^{(n-\ell)/2}(-1)^{b}{n-b
\choose b}f_{k}(n-2b,\ell) =
\sum_{b=0}^{\lfloor n/2\rfloor}(-1)^{b}{n-b \choose b}
\left\{\sum_{\ell=0}^{n-2b}f_{k}(n-2b,\ell)\right\}
\end{align*}
where $\sum_{\ell=0}^{n-2b}f_{k}(n-2b,\ell)$ is given by eq.~(\ref{E:ww2})
and the proof of the theorem is complete.
\end{proof}
\begin{tabular}{c|ccccccccccccccccccccc}
$n$ & \small{1} & \small{2} & \small{3} &\small{4} & \small{5} &
\small{6} & \small{7} & \small{8} & \small{9} & \small{10} &
\small{11} & \small{12} & \small{13} & \small{14} &
\small{15}\\
\hline $S_{3}(n)$ & \small{1} & \small{1} & \small{2} &\small{5} &
\small{13} &\small{36} & \small{105} & \small{321} & \small{1018} &
\small{3334} &\small{11216} &\small{38635} &\small{135835}
&\small{486337} &\small{1769500}\\
\end{tabular}

{\small {\bf Table 1.} The first $15$ numbers of $3$-noncrossing RNA
structures.}

A first implication of Theorem~\ref{T:cool1} is a new proof
for Waterman's formula \cite{Waterman:94a} for the number of
RNA secondary structures with exactly $k$ base pairs or
equivalently having $\ell=n-2k$ isolated vertices.

\begin{corollary}\label{C:k=2}
The number of RNA secondary structures having exactly $\ell$ isolated
vertices, ${\sf S}_2(n,\ell)$, is given by
\begin{equation}\label{E:Waterman-tree}
{\sf S}_2(n,\ell)  =
      \frac{2}{n-\ell}{\frac{n+\ell}{2} \choose \frac{n-\ell}{2} +1}
                    {\frac{n+\ell}{2}-1 \choose \frac{n-\ell}{2}-1} \  .
\end{equation}
Furthermore ${\sf S}_2(n,\ell)$ satisfies the recursion
\begin{equation}\label{E:rec2}
(n-\ell)(n-\ell+2)\cdot {\sf S}_{2}(n,\ell)\, -\, (n+\ell)(n+\ell-2)\cdot
{\sf S}_{2}(n-2,\ell)=0 \ .
\end{equation}
\end{corollary}
\begin{proof}
We actually give two independent proofs of eq~(\ref{E:Waterman-tree}):
the first being a direct computation based on eq.~(\ref{E:da}) and the
second using the recursion in eq.~(\ref{E:rec2}) derived by Zeilberger's
algorithm \cite{Zeilberger:96a}. Let $\frac{n-\ell}{2}=k$ we compute
\begin{align*}
{\sf S}_2(n,\ell)
       &=\sum_{b=0}^{(n-\ell)/2}(-1)^b\binom{n-b}{b}\binom{n-2b}{\ell}
                   \,  C_{\frac{n-\ell-2b}{2}}\\
&=\sum_{b=0}^{k}(-1)^{b}\frac{(n-b)!}{b!(n-2k)!}\cdot
\frac{1}{(k-b+1)!(k-b)!}\\
&=\frac{(n-k-1)!}{(n-2k)!\cdot
k!}\sum_{b=0}^{k}(-1)^{b}\frac{(n-b)!}{(k-b+1)!(n-k-1)!}
\frac{k!}{b!(k-b)!}\\
&=\frac{1}{n-k}{n-k \choose k}\sum_{b=0}^{k}(-1)^{b}{k \choose
b}{n-b \choose k-b+1}\\
&=(-1)^{k+1}\frac{1}{n-k}{n-k \choose k}\sum_{b=0}^{k}{k \choose
b}{k-n \choose k-b+1}\\
&=\frac{1}{n-k}{n-k \choose k+1}{n-k \choose k}\\
&=\frac{1}{k}{n-k \choose k+1}{n-k-1 \choose k-1}.
\end{align*}
As for the second proof we use ${\sf S}_2(n,\ell)=\sum_{b=0}^{(n-\ell)/2}
(-1)^b\binom{n-b}{b}f_2(n-2b,\ell)$ as
the input for Zeilberger's algorithm \cite{Zeilberger:96a} and obtain that
${\sf S}_{2}(n,\ell)$ satisfies the recursion formula
\begin{equation}\label{E:rec21}
(n-\ell)(n-\ell+2)\cdot {\sf S}_{2}(n,\ell)-(n+\ell)(n+\ell-2)\cdot
{\sf S}_{2}(n-2,\ell)=0 \ .
\end{equation}
Using a bijection between RNA secondary structures and linear trees
Waterman computed in \cite{Waterman:94a} the number of RNA secondary
structures with exactly $h$ arcs, $s(n,h)$
\begin{equation}\label{E:water}
s(n,h)=\frac{1}{h}{n-h \choose h+1}{n-h-1 \choose h-1} \ .
\end{equation}
It follows by direct computation that ${\sf S}_2(n,n-2h)=s(n,h)=
\frac{1}{h}{n-h \choose h+1}{n-h-1 \choose h-1}$ satisfies the
recursion in eq.~(\ref{E:rec21}), from which we can conclude
${\sf S}_2(n,\ell) =s(n,(n-\ell)/2)$.
\end{proof}

\begin{corollary}\label{C:k=3}
The number of $3$-noncrossing RNA structures having exactly $\ell$
isolated vertices, ${\sf S}_3(n,\ell)$, is given by
\begin{equation}\label{E:k=3ex}
{\sf S}_3(n,\ell)  =  \sum_{b=0}^{(n-\ell)/2}(-1)^b\binom{n-b}{b}
 \binom{n-2b}{\ell} \left[C_{\frac{n-\ell-2b}{2}}\,
C_{\frac{n-\ell-2b}{2}+2}-C_{\frac{n-\ell-2b}{2}+1}^2\right] \ .
\end{equation}
\end{corollary}

Using the expression of Corollary~\ref{C:k=3} for ${\sf S}_3(n,\ell)$ as
an input for Zeilberger's algorithm \cite{Zeilberger:96a} we derive

\begin{corollary}\label{C:recursion}
 The number of $3$-noncrossing RNA structures having
exactly $\ell$ isolated vertices, ${\sf S}_3(n,\ell)$, satisfies the
$4$-term recursion
\begin{align}
{\sf p}_1(n)\, {\sf S}_{3}(n-6,\ell)-{\sf p}_2(n)\, {\sf S}_{3}(n-4,\ell)-
{\sf p}_3(n){\sf S}_{3}(n-2,\ell)+{\sf p}_4(n)\,{\sf S}_{3}(n,\ell)=0 \ ,
\end{align}
where the coefficients ${\sf p}_1(n,\ell)$, ${\sf p}_2(n,\ell)$
${\sf p}_3(n,\ell)$ and ${\sf p}_4(n,\ell)$ are given by
\begin{eqnarray*}
{\sf p}_1(n,\ell) & = &
\frac{1}{2}n(n-1)(n-10+\ell)(n-4+\ell)(n-8+\ell) \\
{\sf p}_2(n,\ell) & = & \frac{1}{2}n(n-3)(13n^{3}-126n^{2}+13n^{2}\ell-
              88n\ell+392n+3n\ell^{2}+216\ell-384-42\ell^{2}+3\ell^{3}) \\
{\sf p}_3(n,\ell) & = & (n-1)(\frac{1}{2}n-2)(13n^{3}-30n^{2}-13n^{2}\ell+
             8n+16n\ell+3n\ell^{2}+30\ell^{2}-72\ell-3\ell^{3}) \\
{\sf p}_4(n,\ell) & = &
(n-3)(\frac{1}{2}n-2)(n-\ell)(n-\ell+6)(n-\ell+4) \ .
\end{eqnarray*}
\end{corollary}

Theorem~\ref{T:cool1} immediately allows us to
derive the generating function for circular $k$-noncrossing RNA
structures.
Circular RNA structures are $k$-noncrossing digraphs without arcs of the
form $(1,n)$, representing molecular structures over circular sequences.
In circular sequences the arc $(n,1)$ is considered a bond of the primary
sequence and consequently does not occur as an arc in the corresponding
digraph representation. Suppose $k\ge 2$ and let $\mathcal{G}^{(c)}_{n,k}
(\ell,j)$ be the set of all $k$-noncrossing digraphs having exactly $\ell$
isolated points and exactly $j$ $1$-arcs, where a $1$-arc is an arc of the
form $(i,i+1)$, where $i$ is considered modulo $n$.
We set ${\sf G}^{(c)}_{k}(n,\ell,j)=\vert \mathcal{G}^{(c)}_{n,k}
(\ell,j)\vert$.

\begin{theorem}\label{T:circular}
Let $k\in\mathbb{N}$, $k\ge 2$, then the number of circular
$k$-noncrossing RNA structures, with exactly $\ell$ isolated vertices
${\sf S}^{(c)}_k(n,\ell)$, is given by
\begin{equation}\label{E:dacg}
{\sf S}^{(c)}_k(n,\ell) = \sum_{b=0}^{(n-\ell)/2}
(-1)^b \left[\binom{(n-2)-(b-1)}{b-1}+\binom{n-b}{b}\right]
f_k(n-2b,\ell)  \  .
\end{equation}
where $\lambda^{(c)}(n,b)(1,0)=0$, $\lambda^{(c)}(n,b)(1,1)=1$,
$\lambda^{(c)}(2,0)=0$ and  $\lambda^{(c)}(2,2)=1$.
\begin{equation}\label{E:sumcirc}
{\sf S}^{(c)}_k(n)
=\sum_{b=0}^{\lfloor n/2\rfloor}(-1)^{b}\left[\binom{(n-2)-(b-1)}{b-1}+
\binom{n-b}{b}\right]\left\{\sum_{\ell=0}^{n-2b}f_{k}(n-2b,\ell)\right\}
\end{equation}
where $\sum_{\ell=0}^{n-2b}f_{k}(n-2b,\ell)$ is given by eq.~{\rm
(\ref{E:ww2})}.
\end{theorem}
\begin{proof}
For circular RNA structures the $1$-arcs are considered modulo $n$.
Again we derive a family $\mathcal{F}$ of $\mathcal{G}_{n,k}$-digraphs,
having exactly $\ell$ isolated points and at least $b$ $1$-arcs.
We select {\sf (a)} $b$ $1$-arcs, and {\sf (b)} an arbitrary
$k$-noncrossing digraph with exactly $\ell$ isolated points over the
remaining $n-2b$ vertices. In complete analogy we derive that each element
$\theta\in \mathcal{F}$ is contained in $\mathcal{G}^{(c)}_{n,k}(\ell,j)$
for some $j\ge b$.
Let $\lambda^{(c)}(n,b)$, denote the number of ways to select $b$
$1$-arcs over $\{1,\dots,n\}$ including the arc $(n,1)$. Then
$\lambda^{(c)}(n,b)$ is given by
\begin{equation}\label{E:circular}
\lambda^{(c)}(n,b)=
\binom{(n-2)-(b-1)}{b-1}+\binom{n-b}{b} \ ,
\end{equation}
where $\lambda^{(c)}(n,b)(1,0)=0$, $\lambda^{(c)}(n,b)(1,1)=1$,
$\lambda^{(c)}(2,0)=0$ and  $\lambda^{(c)}(2,2)=1$.
Indeed, either the arc $(n,1)$ is selected in which case we are left with
exactly $\binom{(n-2)-(b-1)}{b-1}$ ways to select the remaining $1$-arcs
or $(n,1)$ is not selected, in which case according to
Theorem~\ref{T:cool1}
there are exactly $\binom{n-b}{b}$ ways to select the $1$-arcs. Therefore
we obtain
\begin{equation}\label{E:for}
\sum_{j\ge b}\binom{j}{b}\, {\sf G}^{(c)}_k(n,\ell,j)=
\left[\binom{(n-2)-(b-1)}{b-1}+\binom{n-b}{b}\right]
\, f_k(n-2b,\ell) \ .
\end{equation}
In complete analogy to the argument in Theorem~\ref{T:cool1} we can
conclude
\begin{equation}
{\sf S}^{(c)}_k(n,\ell)=\sum_{b=0}^{(n-\ell)/2}\, (-1)^b\,
   \left[\binom{(n-2)-(b-1)}{b-1}+\binom{n-b}{b}\right]f_k(n-2b,\ell) \ .
\end{equation}
Eq.~(\ref{E:sumcirc}) follows analogously and the proof of the theorem is
complete.
\end{proof}

\section{Restricted RNA structures}\label{S:2-arcs}

We now generalize the ideas in Section~\ref{S:RNA} for the enumeration
of restricted RNA structures.
A restricted RNA structure is an RNA structure without any $2$-arcs,
i.e.~arcs of the form $(i,i+2)$.
In this case we need the condition $k>2$
instead of $k\ge 2$, since our construction can produce $2$-sets of
mutually crossing arcs.
Let $\mathcal{G}_{n,k}(\ell,j_1,j_2)$ be the set of all $k$-noncrossing
digraphs having exactly $\ell$ isolated points and exactly $j_1$ and
$j_2$ $1$-and $2$-arcs. We set
$
{\sf G}_{k}(n,\ell,j_1,j_2)=\vert \mathcal{G}_{n,k}(\ell,j_1,j_2)\vert
$.
In particular we have ${\sf G}_k(n,\ell,0,0)={\sf S}_k^{(r)}(n,\ell)$.

\begin{theorem}\label{T:cool2}
Let $k\in\mathbb{N}$, $k>2$. Then the numbers of restricted RNA structures
${\sf S}_k^{(r)}(n,\ell)$ and ${\sf S}^{(r)}_k(n)$ are given by
\begin{eqnarray}\label{E:da2}
{\sf S}_k^{(r)}(n,\ell) & = &
\sum_{b_{1}\ge 0,b_{2}\ge 0}(-1)^{b_{1}+b_{2}}
\lambda(n,b_{1},b_{2})f_{k}(n-2(b_{1}+b_{2}),\ell) \\
\label{E:da3}
{\sf S}_k^{(r)}(n) & = & \sum_{b_{1}\ge 0,b_{2}\ge 0}^{\lfloor n/2\rfloor}
(-1)^{b_{1}+b_{2}} \lambda(n,b_1,b_2)
\left\{\sum_{\ell=0}^{n-2(b_1+b_2)}f_{k}(n-2(b_{1}+b_{2}),\ell)\right\}\ .
\end{eqnarray}
Here $\lambda(n,b_1,b_2)$ satisfies the recursion
\begin{equation}\label{E:hh}
\lambda(n,b_1,b_2)  =  \lambda(n-2,b_1-1,b_2)
+\lambda(n-1,b_1,b_2) +\lambda(n-4,b_1,b_2-2) +
\lambda(n-3,b_1,b_2-1)
\end{equation}
and the initial conditions for eq.~{\rm (\ref{E:hh})} are
$\lambda(n,0,0)=1$,
$\lambda(n,b_{1},0)={n-b_{1} \choose b_{1}}$, $\lambda(n,0,b_{2})=
\gamma(n,b_{2})$ and $\gamma(n,1)=0$ for $n=1$, $\gamma(n,1)=n-2$ for
$n\ge 2$ and $\gamma(n,2)=0$ for $n=2,3$.
\end{theorem}
\begin{proof}
Suppose $\lambda(n,b_1,b_2)$ is the number of ways to
select exactly $b_1$ $1$-arcs and $b_2$ $2$-arcs over $\{1,\dots,n\}$
vertices.\\
{\it Claim.} $\lambda(n,b_1,b_2)$ satisfies the recursion of
eq.~(\ref{E:hh}) with the respective initial conditions, and we have
\begin{equation}\label{E:form2}
\sum_{j_1\ge b_1,j_2\ge b_2}\binom{j_1}{b_1}\,\binom{j_2}{b_2}\,
{\sf G}_k(n,\ell,j_1,j_2)=
\lambda(n,b_1,b_2)\, f_k(n-2(b_1+b_2),\ell) \ .
\end{equation}
In analogy to the proof of Theorem~\ref{T:cool1} we derive a family
$\mathcal{F}$ of $\mathcal{G}_{n,k}$-digraphs, having exactly $\ell$
isolated points and at least $b_1$ and $b_2$ $1$-arcs and $2$-arcs,
respectively.
We first prove that this construction generates elements of
$\mathcal{G}_{n,k}(\ell,j_1,j_2)$ and then express
$\vert\mathcal{F}\vert$ via the numbers ${\sf G}_k(n,\ell,j_1,j_2)$.
We select {\sf (a)} $b_1$ $1$-arcs
and $b_2$ $2$-arcs and {\sf (b)} an arbitrary $k$-noncrossing digraph
over the remaining $n-2(b_1+b_2)$ vertices with exactly $\ell$ isolated
points. Let $\mathcal{F}$ be the family of digraphs obtained this way. \\
{\it Claim $1$.} Each element $\theta\in \mathcal{F}$ is contained in
$\mathcal{G}_{n,k}(\ell,j_1,j_2)$ for some $j_1\ge b_1$ and $j_2\ge b_2$.\\
To prove this we observe that any $1$-arc or $2$-arc can only cross at
most one other arc. Therefore $1$-arcs and $2$-arcs cannot be contained
in a set of more than $2$-mutually crossing arcs. As a result, for $k>2$
we generate digraphs that are $k$-noncrossing. Clearly $\theta$ has
exactly
$\ell$ isolated vertices and in step {\sf (b)} we potentially derive
additional $1$-arcs and $2$-arcs, whence $j_1\ge b_1$ and $j_2\ge b_2$,
respectively.\\
{\it Claim $2$.}
\begin{equation}\label{E:fam2}
\vert \mathcal{F}\vert =  \lambda(n,b_1,b_2)\, f_k(n-2(b_1+b_2),\ell)\ .
\end{equation}
We prove that the number of ways to select $1$ and $2$-arcs
satisfies the recursion in eq~(\ref{E:hh}) by induction on $n$.
For the induction step we distinguish the following cases:\\
{\sf Case 1.} The arc $(1,2)$ is selected. Then we have
$\lambda(n-2,b_1-1,b_2)$ ways to select $(b_1-1)$ $1$-arcs and $b_2$
$2$-arcs over the vertices $\{3,\dots,n\}$.\\
{\sf Case 2.} The arc $(1,2)$ not selected. Then we distinguish the
scenarios: $(1,3)$ is selected and $(1,3)$ is not selected. In the latter
case we have $\lambda(n-1,b_1,b_2)$ ways to choose $b_1$ $1$-arcs and
$b_2$ $2$-arcs over the vertices $\{2,\dots,n\}$.
Suppose $(1,3)$ is selected. Then we have either that $(2,4)$ is selected,
in which case we can select the remaining $b_1$ $1$-arcs and $b_2$
$2$-arcs over $\{5,\dots,n\}$ in exactly $\lambda(n-4,b_1,b_2-2)$
different ways.
In case $(2,4)$ is not selected we can freely choose $b_1$ $1$-arcs and
$(b_2-1)$ $2$-arcs over $\{4,\dots,n\}$ i.e.~there are
$\lambda(n-3,b_1,b_2-1)$ ways. Therefore we derive the recursion
$$
\lambda(n,b_1,b_2) =  \lambda(n-2,b_1-1,b_2)+
\lambda(n-1,b_1,b_2) +\lambda(n-4,b_1,b_2-2)+\lambda(n-3,b_1,b_2-1)  \,  .
$$
As for the intial conditions, we have are $\lambda(n,0,0)=1$,
$\lambda(n,b_{1},0)={n-b_{1} \choose b_{1}}$, $\lambda(n,0,b_{2})=
\gamma(n,b_{2})$ and $\gamma(n,1)=0$ for $n=1$, $\gamma(n,1)=n-2$ for
$n\ge 2$ and $\gamma(n,2)=0$ for $n=2,3$.
It remains to select an arbitrary $k$-noncrossing digraph with $\ell$
isolated vertices over $n-2(b_1+b_2)$ vertices. According to
Lemma~\ref{L:walks} the
latter number is given by $f_k(n-2(b_1+b_2),\ell)$,
whence eq.~(\ref{E:fam2}) and Claim $2$ is proved.
In view of the fact that any of the $k$-noncrossing digraphs over
$n-2(b_1+b_2)$ vertices can introduce additional $1$-arcs or
$2$-arcs, we set
$$
\mathcal{F}(j_1,j_2)=\{\theta\in\mathcal{F} \mid \theta\ \text{\rm has
                     exactly $j_1$ $1$-arcs and $j_2$ $2$-arcs} \} \ .
$$
Obviously, we have the partition $\mathcal{F}=\dot\bigcup_{j_1 \ge b_1,
\,j_2\ge b_2}\mathcal{F}(j_1,j_2)$. Suppose $\theta \in
\mathcal{F}(j_1,j_2)$.
According to Claim $1$, $\theta\in \mathcal{G}_{n,k}(\ell,j_1,j_2)$
and furthermore $\theta$ occurs with multiplicity $\binom{j_1}{b_1}$
$\binom{j_2}{b_2}$ in $\mathcal{F}$ since by construction any
$b_1$-element subset of the $j_1$ $1$-arcs and
$b_2$-element subset of the $j_2$ $2$-arcs is counted respectively in
$\mathcal{F}$. Therefore we have
\begin{equation} \label{E:well2}
\vert \mathcal{F}(j_1,j_2)\vert = \binom{j_1}{b_1}\binom{j_2}{b_2}
{\sf G}_k(n,\ell,j_1,j_2)
\end{equation}
and
\begin{eqnarray*}
\sum_{j_1 \ge b_1,\, j_2\ge b_2}
\binom{j_1}{b_1}\binom{j_2}{b_2}{\sf G}_k(n,\ell,j_1,j_2) & = &
\sum_{j_1 \ge b_1,\, j_2\ge b_2} \vert \mathcal{F}(j_1,j_2)\vert \\
& = &  \lambda(n,b_1,b_2)f_k(n-2(b_1+b_2),\ell) \ .
\end{eqnarray*}
We next set $F_{k}(x,y)=\sum_{j_{1}\ge 0}\sum_{j_{2} \ge 0}{\sf
G}_{k}(n,\ell,j_{1},j_{2})x^{j_{1}}y^{j_{2}}$.
Taking the $b_1$-th and $b_2$-th derivatives w.r.t.~$x$ and $y$ we obtain
\begin{eqnarray}\label{E:u2}
\frac{1}{b_1!}\frac{1}{b_2!} F_k^{(b_1,b_2)}(1) =
\sum_{j_1\ge b_1,\,j_2\ge b_2}
\binom{j_1}{b_1}\binom{j_2}{b_2}{\sf G}_k(n,\ell,j_1,j_2)\, 1^{j_1-b_1}
1^{j_2-b_2}
 \ .
\end{eqnarray}
Then we have
\begin{align*}
\sum_{j_{1},j_{2} \ge 0}
{\sf G}_{k}(n,\ell,j_{1},j_{2})x^{j_{1}}y^{j_{2}}
&=\sum_{b_{1}\ge 0,b_{2}\ge 0}\left[\sum_{j_{1}\ge b_{1},j_{2}\ge
b_{2}}{j_{1} \choose b_{1}}{j_{2} \choose
b_{2}}{\sf G}_{k}(n,\ell,j_{1},j_{2})\right](x-1)^{b_{1}}(y-1)^{b_{2}}\\
&=\sum_{b_{1}\ge 0,b_{2}\ge 0}
\lambda(n,b_{1},b_{2})\, f_k(n-2(b_1+b_2),\ell)\,(x-1)^{b_{1}}\,
(y-1)^{b_{2}} \ .\\
\end{align*}
By construction ${\sf G}(n,\ell,0,0)$ is the constant term of the
$F_{k}(x,y)$. That is, the number of k-noncrossing RNA structures
with $\ell$ isolated vertices and no 2-arcs is given by
\begin{equation}
{\sf G}(n,\ell,0,0)=
\sum_{b_{1}\ge 0,b_{2}\ge 0}(-1)^{b_{1}+b_{2}}
\lambda(n,b_{1},b_{2})f_{k}(n-2(b_{1}+b_{2}),\ell)
\end{equation}
and taking the sum over all $\ell$ eq.~(\ref{E:da3}) follows
$$
{\sf S}_k^{(r)}(n)  =  \sum_{b_{1}\ge 0,b_{2}\ge 0}^{\lfloor n/2\rfloor}
(-1)^{b_{1}+b_{2}} \lambda(n,b_1,b_2)
\left\{\sum_{\ell=0}^{n-2(b_1+b_2)}f_{k}(n-2(b_{1}+b_{2}),\ell)\right\}
\ ,
$$
where
$\left\{\sum_{\ell=0}^{n-2(b_1+b_2)}f_{k}(n-2(b_{1}+b_{2}),\ell)\right\}$
is given by eq.~(\ref{E:ww2}) and the proof of the theorem is complete.
\end{proof}

\begin{tabular}{c|ccccccccccccccccccccc}
$n$ & \small{1} & \small{2} & \small{3} &\small{4} & \small{5} &
\small{6} & \small{7} & \small{8} & \small{9} & \small{10} &
\small{11} & \small{12} & \small{13} & \small{14} &
\small{15}\\
\hline $S_{3}^{r}(n)$ & \small{1} & \small{1} & \small{1} &\small{2}
& \small{5} &\small{14} & \small{40} & \small{119} & \small{364} &
\small{1145} &\small{3688} &\small{12139} &\small{40734}
&\small{139071} &\small{482214}\\
\end{tabular}

{\small {\bf Table 3.} The first $15$ numbers of $3$-noncrossing restricted
        RNA structures.}

\section{Appendix}\label{S:appendix}
{\bf Proof of Theorem~\ref{T:tableaux}.}
Suppose we have two shapes $\mu^{i}\subsetneq \mu^{i-1}$ and $T_{i-1}$
is a standard Young tableau of shape $\mu^{i-1}$. We first observe that
there exists a unique $j$ and a unique $T_i$ such that $T_{i-1}$ is
obtained from $T_{i}$ by row-inserting $j$ with the RSK algorithm. \\
Suppose $\mu^{i-1}$ differs from $\mu^{i}$ in the first row. Then
$j$ is the element at the end of the first row in $T_{i-1}$.
Otherwise suppose $\ell$ is the row of the square being removed from
$T_{i-1}$. Remove the square and insert its element $x$ into the
$(\ell-1)$-th row at precisely the position, where the removed element
$y$ would push it down via the RSK-algorithm. That is $y$ is maximal
subject to $y<x$. Since each column is strictly increasing $y$
always exists. Iterating this process results in exactly one element
$j$ being removed from $T_i$ and a new filling of $\mu_{i-1}$,
i.e.~a unique tableau $T_{i-1}$. By construction, inserting
$j$ with the RSK algorithm produces $T_{i-1}$.\\
{\it Claim $1$.} There exists a bijection between the set of oscillating
tableaux of length $n$ and digraphs with vertices of degree $\le 1$.\\
Given an oscillating tableau
$(\mu^i)_{i=0}^n$ ($\mu^i$ differs from $\mu^{i-1}$ by at most one square),
we recursively define a sequence
$(G_{0},T_{0}),(G_{1},T_{1}),\ldots,(G_{n},T_{n})$,
where $G_i$ is a digraph and $T_{i}$
is a standard Young tableau. We define $G_0$ to be the digraph with
empty edge-set and $T_{0}$ to be the empty standard
Young tableau.
The tableau $T_{i}$ is obtained from $T_{i-1}$ and
the digraph $G_{i}$ is obtained from $G_{i-1}$ by the following
procedure: \\
{\sf 1.} {\sf (Insert origins)}
   For $\mu^{i}\supsetneq\mu^{i-1}$, then
   $T_{i}$ is obtained from $T_{i-1}$ by adding the entry $i$ in the
   square $\mu^{i}\backslash \mu^{i-1}$.\\
{\sf 2.} {\sf (Isolated vertices)}
For $\mu^{i} = \mu^{i-1}$ then set $T_i=T_{i-1}$\\
{\sf 3.} {\sf (Remove origins)} For $\mu^{i}\subsetneq \mu^{i-1}$, then
let $T_{i}$ be the unique standard Young tableau of shape $\mu^{i}$
and $j$ be the unique number such that $T_{i-1}$ is obtained from $T_{i}$
by row-inserting $j$ with the RSK algorithm. Then set
$E_{G_{i}}=E_{G_{i-1}}\cup \{(j,i)\}$.\\
Obviously, $G_n$ is a digraph, and the set of $i$ where $\mu^{i} =
\mu^{i-1}$ equals the set of isolated vertices of $G_n$.
By construction each entry $j$ is removed exactly once whence no edges
of the form $(j,i)$ and $(j,i')$ can be obtained. Therefore $G_n$ has
degree $\le 1$ and we have a well defined mapping
\begin{eqnarray*}
\beta\colon
\{(\mu_i)_{i=0}^n\mid (\mu_i)_{i=0}^n\,\text{\rm is an oscillating
tableau} \} \longrightarrow \{G_n\mid G_n \, \text{\rm is a digraph with
degree $\le 1$}\} \ .
\end{eqnarray*}
It is clear from the procedure that $G_n$ is a labeled graph and
$\beta$ is injective. To prove surjectivity we observe that each
digraph $G_n$ induces an oscillating tableau as follows. We set
$\mu_{G_n}^n=\varnothing$ and $T_n=\varnothing$. Starting from
vertex $i=n,n-1,\dots, 1,0$ we derive a sequence of Young tableaux
$(T_n,T_{n-1},\dots,T_0)$ as follows:\\
{\sf I.} If $i$ is an terminus of an $G_n$-arc $(j,i)$ add $j$ via the
         RSK-algorithm to $T_{i}$ set
         $\mu_{G_n}^{i-1}\supsetneq \mu_{G_n}^{i}$ to be the
         shape of $T_{i-1}$ (corresponds to {\sf (3)})\\
{\sf II.} If $i$ is an isolated $G_n$-vertex set $\mu_{G_n}^{i-1}=
         \mu_{G_n}^{i}$ (corresponds to {\sf (2)})\\
{\sf III.} If $i$ is the origin of an $G_n$-arc $(i,k)$ let
         $\mu_{G_n}^{i-1}\subsetneq \mu_{G_n}^{i}$ be the shape of
         $T_{i-1}$,
         the standard Young tableau obtained by removing the square
         containing $i$ (corresponds to {\sf (1)}).   \\
Then we have $\beta((\mu_{G_n})_0^n)=G_n$, whence $\beta$ is surjective.\\
{\it Claim $2$.} $G_n$ is $k$-noncrossing if and only if all shapes
$\mu^i$ in the oscillating tableau have less than $k$ rows.\\
From Claim 1 we know $\beta^{-1}(G_{n})=(\varnothing=\mu^{0},
\mu^{1},\ldots \mu^{n}=\varnothing)$, so it suffices to prove that
the maximal number of rows in the shape set $\beta^{-1}(G_{n})$ is
less than $k$. First we observe that the arcs $(i_{1},j_{1}),\ldots
(i_{\ell},j_{\ell})$ form a $\ell$-crossing of $G_{n}$ if and only
if there exists a tableau $T_{i}$ such that elements
$i_{1},i_{2},\ldots i_{\ell}$ are in the $\ell$ squares of $T_{i}$
and being deleted in increasing order $i_{1}<i_{2}<\ldots i_{\ell}$
afterwards. Next, we will obtain a permutation $\pi_{i}$ from the
entries in each tableau $T_{i}$ recursively as follows:\\
 {\sf 1.} If $T_{i-1}$ is obtained from $T_{i}$ by row-inserting $j$
          with the RSK algorithm, then $\pi_{i-1}=\pi_{i}j$. \\
{\sf 2.} If $T_i=T_{i-1}$, then $\pi_{i}=\pi_{i-1}$.\\
{\sf 3.} If $T_{i-1}$ is obtained from $T_{i}$ by deleting the entry
         $i$, then $\pi_{i-1}$ is obtained from $\pi_{i}$ by deleting
         $i$.\\
If $\pi=r_{1}r_{2}\ldots r_{t}$, then the entries being deleted
afterwards are in the order $r_{t},\ldots r_{2},r_{1}$.\\
Using the RSK algorithm w.r.t.~the permutation $\pi_{i}$, the resulting
row-inserting Young tableau is exactly $T_{i}$. We prove this by
induction in reverse order of the oscillating tableau. It is trivial
for the case $i=n$. Suppose it holds for $j$, $1\le j\le n$. Consider
the above three cases: inserting an element, doing nothing and deleting
an element. In the first case, the assertion is implied the RSK
algorithm in the construction of the oscillating tableau. In the second
case, it holds by the induction hypothesis on step $j$.\\
Now it remains to consider the third case, that is, removing the
entry from $T_{j}$ to get $T_{j-1}$. Write $\pi_{j}=x_{1}x_{2}\ldots
x_{p}jy_{1}y_{2}\ldots y_{q}$ and $\pi_{j-1}=x_{1}x_{2}\ldots
x_{p}y_{1}y_{2}\ldots y_{q}$. In view of step {\sf 3} $j$ is larger than
elements $x_{1},x_{2},\ldots,x_{p},y_{1},\ldots y_{q}$. We need to
prove that the insertion tableau $S_{j-1}$ of $\pi_{j-1}$ by the RSK
algorithm is exactly the same as deleting the entry $j$ in $T_{j}$.
We proceed by induction on $q$. In the case $q=0$, $T_{j}$ is
obtained from $T_{j-1}$ by adding $j$ at the end of the first row.
Suppose the assertion holds for $q-1$, that is
$S_{j-1}(x_{1}x_{2}\ldots x_{p}y_{1}y_{2}\ldots
y_{q-1})=S_{j}(x_{1}x_{2}\ldots x_{p}jy_{1}y_{2}\ldots
y_{q-1})\setminus\text{\rm \fbox{$\,j\,$}}$.
Consider inserting $y_{q}$ into $S_{j-1}$,
via the RSK algorithm.
If the insertion track path never touches the position of $j$, then
$S_{j-1}(x_{1}x_{2}\ldots x_{p}y_{1}y_{2}\ldots
y_{q-1}y_{q})=S_{j}(x_{1}x_{2}\ldots x_{p}jy_{1}y_{2}\ldots
y_{q-1}y_{q})\setminus\text{\rm \fbox{$\,j\,$}}$. Otherwise, if
the insertion path
touched $j$ and pushed $j$ into the next row, then since $j$ is greater
than any other entry, $j$ must be moved to the end of next row and the
push process stops. Accordingly, the
insertion path in $S_{j-1}(x_{1}x_{2}\ldots x_{p}y_{1}y_{2}\ldots
y_{q-1})$ is the
same path as in $S_{j}(x_{1}x_{2}\ldots x_{p}jy_{1}y_{2}\ldots
y_{q-1})$ except the last step moving $j$ to a new position $j$, so
deleting $j$ will get $S_{j-1}(x_{1}x_{2}\ldots
x_{p}y_{1}y_{2}\ldots y_{q-1}y_{q})=S_{j}(x_{1}x_{2}\ldots
x_{p}jy_{1}y_{2}\ldots y_{q-1}y_{q})\setminus\text{\rm \fbox{$\,j\,$}}$.
According to
Schensted's Theorem, for any permutation $\pi$, assume $A$ is the
corresponding insertion Young tableau by using the RSK algorithm on
$\pi$. Then the length of the longest decreasing subsequences of
$\pi$ is the number of rows in $A$, whence the assertion. \\
Now we can prove Claim $2$. A diagraph is a $\ell$-crossing
if and only if there exists a $\pi_{i}$ which has decreasing
subsequence of length $\ell$. And the insertion Young tableau of
$\pi_{i}$ is exactly the same with the labeled oscillating tableau
$T_{i}$. According to Schensted's theorem, $\pi$ has a decreasing
sequence of length $\ell$ if and only if rows of $T_{i}$ is $\ell$.\\
 {\it Claim $3$.} There is a bijection
between oscillating tableaux with at most $k-1$ rows of length $n$
and walks with steps $\pm e_i,0$ which stay in the interior of
$C_{0}$ starting and ending at
$(k-1,k-2,\ldots,1)$.\\
This bijection is obtained by setting for $1\le \ell\le k-1$, $x_\ell$
to be the length of the $\ell$-th row. By definition of standard Young
tableaux, we have $\lambda_{1}\ge \lambda_{2}\ge \ldots \lambda_{n}$
i.e.~the length of each row is weakly decreasing. This property also
characterizes walks that stay within the Weyl-chamber $C_0$, i.e.~where
we have $x_{1}> x_{2}\ldots > x_{k-1}>0$ since a walk from
$(k-1,\ldots 2,1)$ to itself in the interior of $C_{0}$ is a translation
of a walk from the origin to itself in the region $x_{1}\ge x_{2}\ldots
\ge x_{k-1}\ge 0$.
In an oscillating tableau $\mu^i$ differs from $\mu^{i-1}$ by at most one
square and adding or deleting a square in the $\ell$-th row or doing
nothing corresponds to steps $\pm e_{\ell}$ and $0$, respectively.
Since the oscillating tableau is of empty shape, we have
walks from the origin to itself, whence Claim $3$ follows and the proof of
the Theorem is complete. $\square$\\


{\bf Proof of Lemma~\ref{L:touch}.}
To prove the lemma we can w.l.o.g.~assume $C=C_0=\{(x_1,\dots,x_{k-1})\mid
x_1> x_2 >\dots> x_{k-1}>0 \}$. Then the assertion is that every walk
having
steps $\pm e_i,0$ starting at $a=(k-1,k-2,\dots,1)$ that crosses from
inside ${C_0}$ into outside ${C_0}$ intersects one of the sub-spaces
$\langle e_1\rangle$ or $\langle e_j-e_{j-1}\rangle$ for $2\le j\le k-1$.
This is correct since to leave $C$ implies that there exists some $i$
such that $x_i\le x_{i+1}$. Let $s_j$ be minimal
w.r.t.~$a+\sum_h^{j+1}s_{h}\not\in C_0$. Since we have steps $\pm e_i,0$
we conclude $x_{k-1}=0$ or $x_j=x_{j-1}$ for some $2\le j\le k-1$, whence
the lemma. $\square$

{\bf Proof of Theorem~\ref{T:reflect}.}
Totally order the roots of $\Delta$. Let $\Gamma_n^-(a,b)$ be the
number of walks $\gamma$ from $a$ to $b$, $a,b\in\mathbb{Z}^{k-1}$ of
length $n$ using the steps $s$, $s\in \{\pm e_i,0\}$ such that
$\langle \gamma(s_r),\alpha\rangle=0$ for some $\alpha\in \Delta$
(i.e.~the walk intersects with the subspace $\langle
\alpha\rangle$).
According to Lemma~\ref{L:touch} every walk that crosses
from inside ${C}$ into outside ${C}$ touches a wall from which we
can draw two conclusions:
\begin{eqnarray}\label{E:split}
\Gamma_n(a,b) & = &\Gamma_n^{+}(a,b)+\Gamma_n^-(a,b) \\
\text{\rm $\beta\neq {\sf id}$}\quad \Longrightarrow \quad
\Gamma_n(\beta(a),b) & = &\Gamma_n^-(\beta(a),b) \ .
\end{eqnarray}
{\it Claim.} $\sum_{\beta\in {\sf B}_{k-1}}(-1)^{\ell(\beta)} \,
\Gamma_n^-(\beta(a),b)=0$.\\
Let $(s_1,\dots,s_n)$ be a walk from $\beta(a)$ to $b$.
By assumption there
exists some step $s_r$ at which we have $(\gamma_{\beta(a),b}(s_r),\alpha)
=0$, for $\alpha\in \Delta$. Let $\alpha^*$ be the largest root for which
we have $(\gamma_{\beta(a),b}(s_r),\alpha^*)=0$ and
$\beta_{\alpha^*}(x)=x-\frac{2\langle \alpha^*,
x\rangle }{\langle \alpha^*,\alpha^*\rangle}\alpha^*$
its associated reflection (eq.~(\ref{E:Co})).
We consider the walk
\begin{equation}
(\beta_{\alpha^*}(s_1),\dots,\beta_{\alpha^*}(s_r),s_{r+1},\dots,s_n)
\end{equation}
Now by definition $(\beta_{\alpha^*}(s_1),\dots,\beta_{\alpha^*}(s_r),
s_{r+1},\dots,s_n)$ starts at $(\beta_{\alpha^*}\circ \beta)(a)$ and has
sign $(-1)^{\ell(\beta)+1}$ since $\ell(\beta)+1=
\ell(\beta_{\alpha^*}\circ \beta)$.
Therefore to each
element $\gamma_{\beta(a),b}$ of $\Gamma_n^-(\beta(a),b)$ having sign
$(-1)^{\ell(\beta)}$ there exits a
$\gamma_{\beta_{\alpha^*}\beta(a),b}\in\Gamma_n^-(\beta_{\alpha^*}
\beta(a),b)$ with sign $(-1)^{\ell(\beta)+1}$ and the claim follows.
We immediately derive
\begin{eqnarray*}
\sum_{\beta\in {\sf B}_{k-1}}(-1)^{\ell(\beta)}\,
\Gamma_n(\beta(a),b) & = &
\Gamma_n(a,b) +
\sum_{\beta\in {\sf B}_{k-1},\beta\neq {\sf id}} (-1)^{\ell(\beta)} \,
\underbrace{\Gamma_n(\beta(a),b)}_{=\Gamma_n^-(\beta(a),b)}\\
& = & \Gamma_n^+(a,b) + \underbrace{\Gamma_n^{-}(a,b)+\sum_{\beta\in
{\sf B}_{k-1},\beta\neq {\sf id}}
                           (-1)^{\ell(\beta)} \, \Gamma_n^-(\beta(a),b)}_{
 \sum_{\beta\in {\sf B}_{k-1}}(-1)^{\ell(\beta)} \,
\Gamma_n^-(\beta(a),b)=0} \ ,
\end{eqnarray*}
whence the theorem. $\square$

{\bf Proof of Lemma~\ref{L:walks}.}
Let $u_i$, $1\le i\le k-1$ be transcendent variables and
$u=(u_i)_1^{k-1}$.
We define
$
{u}^{b-a}=\prod_{i=1}^{k-1} u_{i}^{b_{i}-a_{i}}
$.
Let $F(x,u)$ be a generating function, then $F(x,u)|_{{u}^{b-a}}$
equals the family of coefficients $a_{i}(u)$ at ${u}^{b-a}$ of
$\sum_{i\ge 0}a_i(u)x^i$. We first observe
$$
\Gamma_{n}(a,b)=
\left[1+\sum_{i=1}^{n}(u_{i}+u_{i}^{-1})\right]^n\bigg|_{{u}^{b-a}}
$$
The exponential generating function for $\Gamma_{n}(a,b)$ is
\begin{align*}
\sum_{n \ge 0}\Gamma_{n}(a,b)\frac{x^{n}}{n!}&=\sum_{n \ge 0}
\left[1+\sum_{i=1}^{k-1}(u_{i}+u_{i}^{-1})\right]^n
\bigg|_{{u}^{b-a}}\frac{x^{n}}{n!}\\
&=\sum_{n \ge 0}
\frac{[1+\sum_{i=1}^{k-1}(u_{i}+u_{i}^{-1})]^{n}}{n!}x^{n}
\bigg|_{{u}^{b-a}}\\
&=e^{x}\cdot{\sf exp}[x\sum_{i=1}^{k-1}(u_{i}+u_{i}^{-1})]
\bigg|_{{u}^{b-a}}\\
&=e^{x}\cdot \prod_{i=1}^{k-1}\left({\sf exp}(x(u_{i}+u_{i}^{-1}))
\bigg|_{{u_i}^{b_i-a_i}}\right)\\
\end{align*}
We furthermore derive
\begin{align*}
\sum_{n \ge0}\Gamma_{n}^{+}(a,b)\frac{x^{n}}{n!}& =e^{x}\sum_{\beta \in
B_{k-1}}(-1)^{l(\beta)}\prod_{i=1}^{k-1}
{\sf exp}(x(u_{i}+u_{i}^{-1}))\bigg|_{{u}^{b-\beta(a)}}\\
\end{align*}
and writing $\beta=\epsilon_{h}\,\sigma$ we obtain
\begin{align*}
\sum_{n \ge 0}\Gamma_{n}^{+}(a,b)\frac{x^{n}}{n!}&=
e^{x}\sum_{\sigma \in S_{k-1}}\sum_{h=1}^{k-1}\eta_{h}\,{\rm sgn}
(\sigma)\prod_{i=1}^{k-1}\left({\sf exp}(x(u_{i}+u_{i}^{-1}))
\bigg|_{u_{i}^{b_{i}-\epsilon_{h}a_{\sigma_{i}}}}\right)\\
&=e^{x}\sum_{\sigma \in S_{k-1}}{\rm sgn}(\sigma)\sum_{h=1}^{k-1}
\eta_h\prod_{i=1}^{k-1}\left({\sf exp}(x(u_{i}+u_{i}^{-1}))
\bigg|_{u_{i}^{b_{i}-\epsilon_{h}a_{\sigma_{i}}}}\right)\\
&=e^{x}\sum_{\sigma \in S_{k-1}}{\rm sgn}(\sigma) \left\{
\prod_{i=1}^{k-1}\left({\sf exp}(x(u_{i}+u_{i}^{-1}))
\bigg|_{u_{i}^{b_{i}-a_{\sigma_{i}}}}\right)\right.\\
& \qquad \qquad \qquad \qquad
\left.-\prod_{i=1}^{k-1}\left({\sf exp}(x(u_{i}+u_{i}^{-1}))
\bigg|_{u_{i}^{b_{i}+a_{\sigma_{i}}}}\right)\right\}
\\
&=e^{x}{\sf det}_{{k-1\times k-1}}
            [I_{a_{i}-b_{j}}(2x)-I_{a_{i}+b_{j}}(2x)]|_{i,j=1}^{k-1}\\
\end{align*}
where $\eta_h=\pm 1$ and the lemma follows. $\square$

{\bf Acknowledgments.}
We are grateful to Prof.~W.Y.C.~Chen and Prof.~Xin for helpful comments.
Many thanks to J.Z.M.~Gao and F.W.D.~Huang for their help and L.C.~Zuo
for her suggestions.
This work was supported by the 973 Project, the PCSIRT Project of the
Ministry of Education, the Ministry of Science and Technology, and
the National Science Foundation of China.

\bibliography{bi}
\bibliographystyle{plain}


\end{document}